\documentclass[12pt]{amsart}
\usepackage{amssymb}
\usepackage{epsf, epsfig}
\usepackage{colordvi}

\copyrightinfo{}{}

\usepackage{amssymb}
\usepackage{epsf, epsfig}
\usepackage{colordvi}

\newtheorem{theorem}{Theorem}[section]
\newtheorem{lemma}[theorem]{Lemma}

\newtheorem{proposition}[theorem]{Proposition}
\newtheorem{defn}[theorem]{Definition}
\newtheorem{ex}[theorem]{Example}

\newenvironment{definition}{\begin{defn}\rm}{\end{defn}}
\newenvironment{example}{\begin{ex}\rm}{\end{ex}}

\theoremstyle{remark}

\numberwithin{equation}{section}

\newcommand{\C}{\mathbb{C}}
\newcommand{\Z}{\mathbb{Z}}
\newcommand{\Fl}{\mathcal{F}\!\ell}
\newcommand{\Gr}{\mbox{\rm Gr}}

\newcommand{\ptclass}{[\mathrm{pt}]}
\renewcommand{\Box}{{\includegraphics[height=7pt]{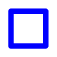}}}
\newcommand{\Rect}{{\includegraphics[height=8pt]{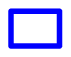}}}
\newcommand{\Shift}{{\includegraphics[height=8pt]{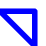}}}

\newcommand{\QED}{{\includegraphics[height=10pt]{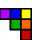}}}

\begin{document}

\title{A Littlewood-Richardson rule for Grassmannian permutations}  


\author{Kevin Purbhoo}
\address{Combinatorics \& Optimization\\
         University of Waterloo\\
         Waterloo, ON, N2L 3G1\\
         CANADA}
\email{kpurbhoo@math.uwaterloo.ca}
\urladdr{http://www.math.uwaterloo.ca/\~{}kpurbhoo/}

\author{Frank Sottile}
\address{Department of Mathematics\\
         Texas A\&M University\\
         College Station\\
         TX \ 77843\\
         USA}
\email{sottile@math.tamu.edu}
\urladdr{http://www.math.tamu.edu/\~{}sottile}
\thanks{Work of Sottile supported by NSF
         CAREER grant DMS-0538734} 

%
%

\date{}



\begin{abstract}
 We give a combinatorial rule for computing intersection numbers on a 
 flag manifold which come from products of Schubert classes pulled back from
 Grassmannian projections.
 This rule generalizes the known rule for Grassmannians.
\end{abstract}
\maketitle

%
%
%
\section*{Introduction}

One of the main open problems in Schubert calculus is to find an analog of the
Little\-wood-Richardson rule for flag manifolds~\cite[Problem~11]{St00}, and more
generally to find combinatorial formulae for intersection numbers of Schubert varieties. 
This problem was recently solved by Coskun for two-step flag manifolds~\cite{Co07}.  

We give such a combinatorial interpretation for intersection numbers of
Grassmannian Schubert problems on any type $A$ flag manifold.
This number counts certain objects that we call filtered tableaux which 
satisfy conditions coming from the Schubert problem.
When the flag manifold is a Grassmannian this coincides with a standard interpretation
of these numbers obtained from the Littlewood-Richardson rule.
Grassmannian Schubert problems on the flag manifold were studied 
in \cite{RSSS06}; they are exactly the Schubert problems which appear 
in the generalization of the Shapiro conjecture to flag manifolds
given there.

In Section~\ref{Sec:Rule} we define filtered tableaux, give an example, 
and state our formula, which we prove in Section~\ref{Sec:Proofs}.
Our proof uses some identities of~\cite{BS98} which were established using geometry, and
is thus not completely combinatorial.
In Section~\ref{Sec:Comments} we explain how our formula relates to one coming from Monk's
formula~\cite{Mo59} and discuss how to give a purely combinatorial 
proof based on
the rule of Kogan~\cite{Ko01}. 

We thank the Centre de recherches math\'ematiques in Montr\'eal and the 
organizers of the workshop on Combinatorial Hopf Algebras and Macdonald 
Polynomials in May 2007 where this collaboration was begun.  
We also thank Hugh Thomas for helpful discussions, and Chris Hillar
for comments on the manuscript.

%
%
%
\section{A Littlewood-Richardson rule for Grassmannian Schubert problems}\label{Sec:Rule}

For background on flag manifolds and Schubert calculus, see~\cite{Fu97}.
We fix a positive integer $n$ throughout.
Let $\alpha=\{\alpha_1, \alpha_2, \dotsc, \alpha_m\}$
be a non-empty subset of $[n{-}1]:=\{1,2,\dotsc,n{-}1\}$,
which we write in increasing order
\[
   \alpha\ \colon\ 0=\alpha_0<\alpha_1<\dotsb<\alpha_m<\alpha_{m+1}=n\,.
\]
A \Blue{{\it partial flag of type $\alpha$}} is a sequence $F_\bullet$ of linear subspaces 
in $\C^n$
\[
  F_\bullet\ \colon\ \{0\}\ \subset\ F_1\ \subset\ F_2\ \subset\ 
    \dotsb\ \subset\ F_m\ \subset\ \C^n\,,
\]
where $\dim F_i=\alpha_i$.
The set \Blue{$\Fl_\alpha$} of all flags of type $\alpha$ is a complex manifold of dimension
\[
   \Blue{\dim(\alpha)}\ :=\ \sum_{i=1}^m (n-\alpha_i)(\alpha_i-\alpha_{i-1})\,.
\]
Schubert varieties and classes in $\Fl_\alpha$ are indexed by permutations 
$w$ of $\{1,2, \dotsc, n\}$ whose descent set is contained in $\alpha$. 
For a permutation $w$, let $\sigma_w$ be the class of the Schubert variety 
corresponding to $w$, following the conventions in~\cite{Fu97}.
Its cohomological degree is $2\ell(w)$, where $\ell(w)$ counts the number of 
inversions $\{i<j\mid w(i)>w(j)\}$ of $w$.

If $\beta\subset\alpha$ is another subset then there is a
projection $\pi_{\alpha,\beta}\colon \Fl_\alpha\to \Fl_\beta$ whose fibres are products of 
flag varieties.
When $\beta=\{b\}$ is a singleton, $\Fl_\beta$ is the Grassmannian $\Gr(b,n)$ of
$b$-planes in $\C^n$.
In this case, we write $\pi_b$ for $\pi_{\alpha,\beta}$.
We note that $\pi_{\alpha,\beta}^* \sigma_w$ is just the Schubert class
$\sigma_w \in H^*(\Fl_\beta)$.

Schubert classes in $\Gr(b,n)$ are also indexed by partitions
$\lambda$, which are northwest-justified arrays of boxes
in a $b\times(n-b)$ rectangle, $\Rect_b$.  
Associated to a partition
$\lambda$ is the \Blue{{\it Grassmannian permutation}} 
$w$ with \Blue{{\it shape}} $\lambda$ and descent at $b$.
This permutation has a unique descent 
at $b$,
and its first $b$ values are 
\[
  w(i)\ =\ i+\lambda(b+1-i)
  \qquad\mbox{for}\quad i=1,\dotsc,b\,.
\]
Here, $\lambda(i)$ denotes the number of boxes in row $i$ of $\lambda$.  
We write $\sigma_\lambda$ for the Grassmannian Schubert class $\sigma_w$.
Here are three partitions with $b=3$ and $n=7$;
the third is also drawn inside $\Rect_3$.
They correspond to the Grassmannian permutations
$1352467$, $1372456$, and $2471356$.
%
%
%
%
%
\[
  \includegraphics{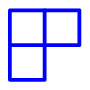}\qquad
  \includegraphics{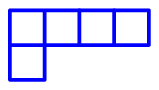}\qquad
  \includegraphics{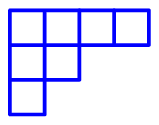}\qquad
  \includegraphics{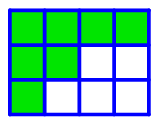}
\]
Let $|\lambda|$ be the number of boxes in $\lambda$.
This is half the cohomological degree of the Schubert class $\sigma_\lambda$ 
and is the complex codimension of the associated Schubert variety.\smallskip

The Littlewood-Richardson rule for the Grassmannian 
expresses a product
$\sigma_\lambda\cdot\sigma_\mu$ of two Schubert classes 
as a sum of classes $\sigma_\nu$ where $\lambda,\mu\subset\nu$ with
$|\nu|=|\mu|+|\lambda|$.
In this rule, the coefficient 
\Blue{$c^{\nu/\mu}_\lambda$}
of $\sigma_\nu$ is the
number of \Blue{{\it Littlewood-Richardson tableaux}} of skew shape $\nu/\mu:=\nu-\mu$ and   
content $\lambda$.
These are fillings of the boxes in $\nu/\mu$ with positive integers such that
 \begin{enumerate}

  \item[(i)]
          The entries weakly increase left-to-right across each row and strictly increase
          down each column.

  \item[(ii)]
         The number of $j$s in the filling is equal to $\lambda(j)$,
         the number of boxes in row $j$ of $\lambda$.

  \item[(iii)] 
         If we read the entries right-to-left across each row and from the top row to the
         bottom row, then at every step we will have encountered at least as many
         occurrences of $i$ as of $i{+}1$ for each positive integer $i$.
 \end{enumerate}

For example, here are some Littlewood-Richardson tableaux.
 \begin{equation}\label{Eq:LR_rule}
   \raisebox{-25pt}{$
   \begin{picture}(52,52)(0,-6.5)
    \put(0,0){\includegraphics[height=39pt]{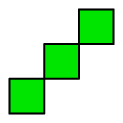}}
    \put(3.5,2.5){2} \put(17,15.5){1} \put(30,28.5){1}
   \end{picture}
    \quad
   \begin{picture}(52,52)(0,-6.5)
    \put(0,0){\includegraphics[height=39pt]{figures/LR1.eps}}
    \put(4,2.5){1} \put(16.5,15.5){2} \put(30,28.5){1}
   \end{picture}
    \quad \qquad
   \begin{picture}(52,52)
    \put(0,0){\includegraphics[height=52pt]{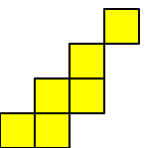}}
                                                     \put(43,41.5){1}
                                    \put(29.5,28.5){2}
                  \put(17  ,15.5){1}\put(29.5,15.5){3}
    \put(4,2.5){1}\put(16.5, 2.5){2}
   \end{picture}
    \qquad
   \begin{picture}(52,52)
    \put(0,0){\includegraphics[height=52pt]{figures/LR2.eps}}
                                                      \put(43,41.5){1}
                                    \put(30  ,28.5){1}
                  \put(16.5,15.5){2}\put(29.5,15.5){2}
    \put(4,2.5){1}\put(16.5, 2.5){3}
   \end{picture}
    \qquad
   \begin{picture}(52,52)
    \put(0,0){\includegraphics[height=52pt]{figures/LR2.eps}}
                                                       \put(43,41.5){1}
                                      \put(30  ,28.5){1}
                    \put(17  ,15.5){1}\put(29.5,15.5){2}
    \put(3.5,2.5){2}\put(16.5, 2.5){3}
   \end{picture}
  $}
 \end{equation}

A \Blue{{\it Grassmannian Schubert class}} in the cohomology ring of 
$\Fl_\alpha$ is the
pullback of a Schubert class along a projection to a Grassmannian.
That is, it has the form $\pi^*_b\sigma_\lambda$ where $b\in\alpha$ and
$\lambda\subset\Rect_b$.
These are indexed by pairs $(b,\lambda)$ with $\lambda\subset\Rect_b$.

A \Blue{{\it Grassmannian Schubert problem}} is a list
$((a_1,\lambda_1),\dotsc,(a_s,\lambda_s))$
with $a_1\leq\dotsb\leq a_s$.
We require that for every $i=1,\dotsc,s$ we have
$a_i\in\alpha$ and $\lambda_i\subset\Rect_{a_i}$, and also 
 \begin{equation}\label{Eq:dimension}
   |\lambda_1| + |\lambda_2| +\dotsb+|\lambda_s|
    \ =\ \dim(\alpha)\,.
 \end{equation}
By the dimension condition~\eqref{Eq:dimension}, we have
 \[
   \prod_{i=1}^s \pi^*_{a_i} \sigma_{\lambda_i}
   \ \in\ H^{2\dim(\alpha)}(\Fl_\alpha)\ =\ 
    \Z\cdot \ptclass_\alpha\,,
 \]
where $\ptclass_\alpha$ is the class of a point in $\Fl_\alpha$.
The problem that we solve is to give
a combinatorial formula for the coefficient
of $\ptclass_\alpha$ in this product.
Note that if $\alpha \supsetneq \{a_1, \dotsc, a_s\}$ this coefficient
is zero (e.g. by \cite[Lemma 1]{Kn00}), and so we will generally
assume that $\alpha = \{a_1, \dotsc, a_s\}$.

Write $\Shift_\alpha$ for the union of all rectangles $\Rect_a$ for 
each $a\in\alpha$, where the rectangles all share the same upper right corner.
Here are three such shapes when $n=7$. 
\[
   \raisebox{30pt}{$\Shift_{235}\ =\ $}\ 
        \includegraphics[height=40pt]{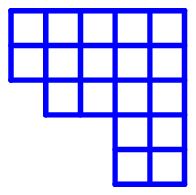}  \quad\qquad
   \raisebox{30pt}{$\Shift_{145}\  =\ $}\ 
        \includegraphics[height=40pt]{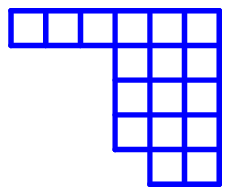}\quad\qquad
   \raisebox{30pt}{$\Shift_{[6]}\ =\ $}\ 
        \includegraphics[height=40pt]{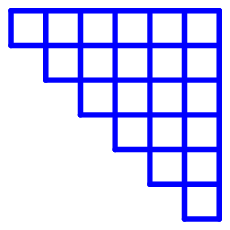}  
\]

A \Blue{{\it shape}} $\mu\subset\Shift_\alpha$ is a subset of boxes which are 
northwest justified. 
For example, when $n=6$, the shaded boxes are four shapes in $\Shift_{234}$.
\[
\includegraphics[height=40pt]{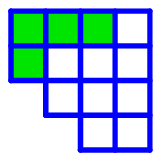}  \quad\qquad
\includegraphics[height=40pt]{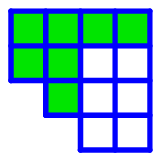}  \quad\qquad
\includegraphics[height=40pt]{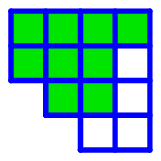}  \quad\qquad
\includegraphics[height=40pt]{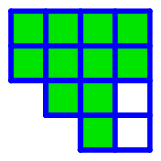}  \quad\qquad
\]

\begin{definition}
Let $\Lambda=((a_1,\lambda_1),\dotsc,(a_s,\lambda_s))$ be a 
Grassmannian Schubert problem.
Set $\alpha=\{a_1,a_2,\dotsc,a_s\}$ and fix a shape $\mu\subset\Shift_\alpha$.
A \Blue{{\it filtered tableau $T_\bullet$  with shape $\mu$ and content $\Lambda$}} 
is a sequence
\[
  \mu_\bullet\ \colon\ \emptyset=\mu_0\ \subset\ \mu_1\ \subset\ 
     \mu_2\ \subset\ \dotsb\ \subset\ \mu_{s+1}\ \subset\ \mu_s=\mu
\]
of shapes together with fillings $T_1,\dotsc,T_s$ of the skew shapes
$\mu_i/\mu_{i-1}$ by positive integers which satisfy 
the following properties.
 \begin{enumerate}

  \item The skew shape $\mu_i/\mu_{i-1}$ must fit entirely
         within the rectangle $\Rect_{a_i}\subset\Shift_\alpha$.

  \item The filling $T_i$ is a Littlewood-Richardson tableau of content $\lambda_i$. 
 \end{enumerate}
Note that we must have $|\mu|=|\lambda_1|+\dotsb+|\lambda_s|$.
\hfill\QED 
\end{definition}

  An induction shows that the coefficient of 
  $\ptclass_b=\sigma_{\includegraphics[height=6pt]{figures/rect.eps}_b}$
  in a product $\sigma_{\lambda_1}\dotsb\sigma_{\lambda_s}$ in 
  $H^*(\Gr(b,n))$ is the number of filtered tableaux with shape
  $\Rect_b$ whose content is the sequence
  $((b,\lambda_1),\dotsc,(b,\lambda_s))$.
  We generalize this to any flag manifold. 

\begin{theorem}\label{Thm:Main}
  Let $\Lambda=((a_1,\lambda_1),\dotsc,(a_s,\lambda_s))$ be a Grassmannian Schubert
  problem on $\Fl_\alpha$.
  Then the coefficient of \/ $\ptclass_\alpha$ in the product
  $\prod_i \pi^*_{a_i}\sigma_{\lambda_i}$ is the number of filtered tableaux with shape 
  $\Shift_\alpha$ and content $\Lambda$. 
\end{theorem}

\begin{example}
 We use this formula to compute the intersection number $N$, defined by
\[
   N\ptclass_\alpha\ =\ 
   \pi_1^*(\sigma_{\includegraphics[height=4pt]{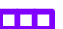}})\cdot
   \pi_2^*(\sigma_{\includegraphics[height=4pt]{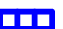}})\cdot
   \pi_3^*(\sigma_{\includegraphics[height=8pt]{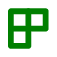}})\cdot
   \pi_4^*(\sigma_{\includegraphics[height=12pt]{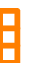}})\cdot
   \pi_5^*(\sigma_{\includegraphics[height=12pt]{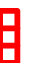}})\,.
\]
 Here, $\alpha=[4]$ and $\Shift_\alpha$ is the full staircase shape.
 There are exactly three sequences of shapes $\mu_\bullet$ which satisfy
 the condition (1) in the definition of filtered tableaux.
\[
  \includegraphics[height=50pt]{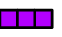}\ \ 
     \raisebox{41pt}{$\subset$\rule{0pt}{13pt}}\ \ 
  \includegraphics[height=50pt]{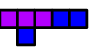}\ \ 
     \raisebox{41pt}{$\subset$}\ \ 
  \includegraphics[height=50pt]{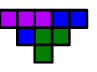}\ \ 
     \raisebox{41pt}{$\subset$}\ \ 
  \includegraphics[height=50pt]{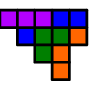}\ \ 
     \raisebox{41pt}{$\subset$}\ \ 
  \includegraphics[height=50pt]{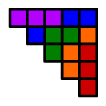}
\]
\[
  \includegraphics[height=50pt]{figures/T1.1.eps}\ \ 
     \raisebox{41pt}{$\subset$}\ \ 
  \includegraphics[height=50pt]{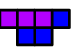}\ \ 
     \raisebox{41pt}{$\subset$}\ \ 
  \includegraphics[height=50pt]{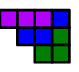}\ \ 
     \raisebox{41pt}{$\subset$}\ \ 
  \includegraphics[height=50pt]{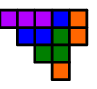}\ \ 
     \raisebox{41pt}{$\subset$}\ \ 
  \includegraphics[height=50pt]{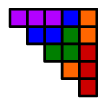}
\]
\[
  \includegraphics[height=50pt]{figures/T1.1.eps}\ \ 
     \raisebox{41pt}{$\subset$}\ \ 
  \includegraphics[height=50pt]{figures/T2.2.eps}\ \ 
     \raisebox{41pt}{$\subset$}\ \ 
  \includegraphics[height=50pt]{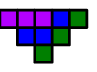}\ \ 
     \raisebox{41pt}{$\subset$}\ \ 
  \includegraphics[height=50pt]{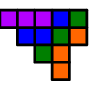}\ \ 
     \raisebox{41pt}{$\subset$}\ \ 
  \includegraphics[height=50pt]{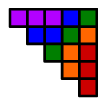}
  \raisebox{-5pt}{\rule{0pt}{0pt}}
\]
 Each of the first two sequences support a unique filtered tableau satisfying condition
 (2), while the third supports two; 
 thus the required intersection number is 4, which may be verified by
 direct computation using the Pieri formula for flag manifolds~\cite{So96}.
 Indeed, there is a unique Littlewood-Richardson tableau of shape $\nu/\mu$ and 
 content $\lambda$ when $\lambda$ is a single row or column and also when 
 the shapes of $\nu/\mu$ and $\lambda$ are the same or rotated by $180^\circ$.
 The only skew shape here which admits more than 
 one Littlewood-Richardson tableau is when  
 $\lambda=\includegraphics[height=10pt]{figures/21.3.eps}$
 and $\nu/\mu=\includegraphics[height=9pt]{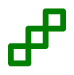}$.
 There are two such Littlewood-Richardson tableaux, given in~\eqref{Eq:LR_rule},
 and this occurs in the middle of the third chain. \hfill\QED 
\end{example}

%
%
%
\section{Proof of Theorem~\ref{Thm:Main}}\label{Sec:Proofs}

Let $\Fl:=\Fl_{[n-1]}$ be the manifold of complete flags in $\C^n$,
which has dimension $\binom{n}{2}$.  Its Schubert classes are
indexed by all
permutations $w$ of the numbers $\{1,2,\dotsc,n\}$.  
We prove a strengthening of Theorem~\ref{Thm:Main} for the full flag manifold and use this
to deduce Theorem~\ref{Thm:Main} for all partial flag manifolds.
We give the key definition of this section.

\begin{definition}
 A permutation $w$ is a \Blue{{\it valley permutation}} with 
 \Blue{{\it floor at $a$}} if 
 \[
    w(1)>w(2)>\dotsb>w(a)
    \qquad\mbox{and}\qquad
   w(a{+}1)<w(a{+}2)<\dotsb<w(n)\,. \eqno{\QED}
 \]
\end{definition}

For example, $53\Red{1}246$ and $64\Red{3}125$ are valley permutations with floor at 3.
We associate a shape $\mu=\mu(w)$ to any valley permutation $w$.
If $w$ has floor at $a$, then $\mu(w)$ is the shape whose rows are
\[
   w(1)-1\ >\ w(2)-1\ >\ \dotsb\ >\ w(a)-1\ \geq\ 0\,.
\]
This has either $a$ or $a{-}1$ rows.
Observe that $w$ is determined by $\mu(w)$ and that $\ell(w)=|\mu(w)|$ where $\ell(w)$
counts the inversions in $w$. 
For example,
\[
    \raisebox{19pt}{$\mu(53\Red{1}246)\ =\ $}\ 
        \includegraphics[height=30pt]{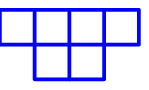}
   \qquad\raisebox{19pt}{and}\qquad
    \raisebox{19pt}{$\mu(64\Red{3}125)\ =\ $}\ 
       \includegraphics[height=30pt]{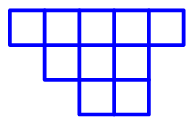}\ 
          \raisebox{19pt}{.}
\]

\begin{theorem}\label{Th:coefficients}
   Let $\Lambda=((a_1,\lambda_1),\dotsc,(a_t,\lambda_t))$ with 
   $a_1\leq a_2\leq\dotsb\leq a_t$ and suppose that $w$ is a valley permutation with
   shape $\mu$. 
   Then the coefficient of $\sigma_w$ in the product
   $\prod_{i=1}^t\pi^*_{a_i}\sigma_{\lambda_i}$ in the cohomology ring of 
   $\Fl$ is the number of filtered tableau with shape $\mu$ and content $\Lambda$.
\end{theorem}

Since the class $\ptclass$ of a point in $H^*(\Fl)$ is indexed by the longest permutation,
which is a valley permutation with shape $\Shift_{[n{-}1]}$, Theorem~\ref{Th:coefficients}
implies Theorem~\ref{Thm:Main} for $\Fl_{[n-1]}$.
We deduce Theorem~\ref{Thm:Main} for general flag manifolds $\Fl_\alpha$ from
the case for $\Fl_{[n-1]}$.

\begin{proof}[Proof of Theorem~\ref{Thm:Main}]
Suppose that $b \notin \alpha$, say $ \alpha_i < b <\alpha_{i+1}$, and set
$\alpha' := \alpha \cup \{b\}$.  
We assume that the theorem holds for $\Fl_{\alpha'}$,
and deduce it for $\Fl_{\alpha}$.

Let $\kappa$ be the rectangular partition
with $b{-}\alpha_i$ rows and $\alpha_{i+1}{-}b$ columns. 
Set 
$\Lambda' := ((a_1,\lambda_1),\dotsc, (b, \kappa), \dotsc, (a_s,\lambda_s))$.
Note that $\pi_{\alpha',b}^* \sigma_\kappa$ is dual
to $\pi_{\alpha',\alpha}^*\ptclass_\alpha$ in $H^*(\Fl_{\alpha'})$
under the Poincar\'e pairing.  
Thus, for any $\tau \in H^*(\Fl_\alpha)$ we have
\[
  \big[\ptclass_{\alpha'}\big]\,
  \pi_{\alpha',b}^* \sigma_\kappa \cdot \pi_{\alpha',\alpha}^* \tau
  \ = \ \big[\ptclass_\alpha\big]\, \tau\,,
\]
where $\big[\ptclass_\alpha\big] \tau$
denotes the coefficient of $\ptclass_\alpha$ in $\tau$.
In particular,
 \begin{equation} \label{Eqn:Induction}
  \big[\ptclass_{\alpha'}\big]
   \prod_{(a, \lambda) \in \Lambda} \pi_a^* \sigma_\lambda
  \ = \ \big[\ptclass_\alpha\big]
   \prod_{(a', \lambda') \in \Lambda'} \pi_{a'}^* \sigma_{\lambda'}\,.
 \end{equation}
There is a bijection between filtered tableaux with shape $\Shift_\alpha$ and
content $\Lambda$ and those with shape $\Shift_{\alpha'}$ and content $\Lambda'$, obtained
by inserting the unique Littlewood-Richard\-son tableau of shape and content $\kappa$ into
the filtration.  
Thus counting either set of filtered tableaux gives the coefficient~\eqref{Eqn:Induction}.
\end{proof}

A Schubert class $\sigma_w$ \Blue{{\it appears}} in a product 
$\sigma_u\dotsb\sigma_v$ of Schubert classes if, when we expand the product in the basis
of Schubert classes, $\sigma_w$ appears with a positive coefficient.

We will prove Theorem~\ref{Th:coefficients} by induction on the number of terms $t$ in the
product. 
Important for this is the following proposition which summarizes some discussion at the 
beginning of Section~1 in~\cite{BS98}.

\begin{proposition}\label{P:a-Bruhat}
   If a Schubert class $\sigma_w$ appears in the product
   $\sigma_v\cdot \pi^*_a\sigma_\lambda$, then  the following conditions hold.
\begin{itemize}
  \item[$(1)$] Whenever $i\leq a<j$, we have $w(i)\geq v(i)$ and 
                  $w(j)\leq v(j)$.
  \item[$(2)$] If $i<j\leq a$ and $v(i)<v(j)$, then $w(i)<w(j)$.
               If $a<i<j$ and $v(i)<v(j)$, then $w(i)<w(j)$.
\end{itemize}
\end{proposition}

In~\cite{BS98}, it is shown that the conditions in Proposition~\ref{P:a-Bruhat} 
define an order relation $v\leq_a w$, which is a suborder of the Bruhat order.
We deduce an important lemma.

\begin{lemma}\label{Lem:descents}
  If $\sigma_w$ appears in $\prod_{i=1}^t\pi^*_{a_i}\sigma_{\lambda_i}$ then 
  $w$ has no descents after position $a_t$.
\end{lemma}

\begin{proof}
 We prove this by induction on $t$.
 It holds when $t=0$, as the multiplicative identity in cohomology
 is the Schubert class indexed by the identity permutation.

 Suppose that $\sigma_w$ appears in the product
 $\prod_{i=1}^t\pi^*_{a_i}\sigma_{\lambda_i}$.
 Then there is some permutation $v$ such that $\sigma_v$ 
 appears in the product $\prod_{i=1}^{t-1}\pi^*_{a_i}\sigma_{\lambda_i}$
 and $\sigma_w$ appears in the product $\sigma_v\cdot\pi^*_{a_t}\sigma_{\lambda_t}$. 
 Hence $v\leq_{a_t} w$.
 Since $v$ has no descents after position $a_{t-1}$ and $a_{t-1}\leq a_t$, 
 condition (2) of Proposition~\ref{P:a-Bruhat} implies that 
 $w$ has no descents after position $a_t$. 
\end{proof}

For permutations $v,w$ and a partition $\lambda\subset\Rect_a$, let
$c^w_{v,a,\lambda}$ be the coefficient of  $\sigma_w$ in
  the product $\sigma_v\cdot\pi^*_a\sigma_\lambda$.
One of the main results in~\cite{BS98} is 
the following identity.

\begin{proposition}\label{P:identity}
  Suppose that $v\leq_a w$ and $x\leq_a z$ with $wv^{-1}=zx^{-1}$.
  Then for every $\lambda\subset \Rect_a$ we have
  $c^w_{v,a,\lambda}=c^z_{x,a,\lambda}$.
\end{proposition}

Suppose that a shape $\nu\subset\Shift_{[n-1]}$ has either $b{-}1$ or $b$ rows.
We define $\nu|_b$ to be the intersection of the shape $\nu$ with $\Rect_b$.

\begin{proof}[Proof of Theorem~\ref{Th:coefficients}]
We proceed by induction on $t$.  The theorem holds (trivially) for $t=0$;
assume that $t>0$ and that it holds for $t-1$.

 Let $w$ be a valley permutation with shape $\mu$, and suppose that $w$ appears in the
 product $\prod_{i=1}^t \pi^*_{a_i}\sigma_{\lambda_i}$.
 Then by Lemma~\ref{Lem:descents}, $w$ has a floor at $a_t$.
 Let us expand the product
\[
   \prod_{i=1}^{t-1} \pi^*_{a_i}\sigma_{\lambda_i}
   \ =\ 
   \sum_v c^v \sigma_v\,.
\]
 Then the coefficient of $\sigma_w$ in the product
 $\prod_{i=1}^t \pi^*_{a_i}\sigma_{\lambda_i}$ is the sum
\[
   \sum_{v\leq_{a_t} w} c^v\cdot  c^w_{v,a_t,\lambda_t}\,.
\]

 Suppose that $v\leq_{a_t} w$.  
 Since $w$ has a floor at $a_t$, Proposition~\ref{P:a-Bruhat}(2) implies that 
\[
   v(1)\ >\ v(2)\ >\ \dotsb\ >\ v(a_t)\,.
\]
 If the coefficient $c^v\neq 0$, so that $v$ can contribute to this sum,
 then  Lemma~\ref{Lem:descents} implies that $v$ has no descents after position $a_{t-1}$.
 Since $a_t{-}1\leq a_{t-1}\leq a_t$, this implies that $v$ is a valley permutation with a
 floor at $a_t$.

 Let $\nu$ be the shape of $v$.
 Since both $w$ and $v$ have floor at $a_t$, both $\mu$ and $\nu$ have
 either $a_t{-}1$ or $a_t$ rows, and
 thus $\mu/\nu\subset\Rect_{a_t}$.
 The theorem would follow if we knew that 
 \begin{equation}\label{Eq:identity}
   c^w_{v,a_t,\lambda_t}\ =\ c^{\mu/\nu}_{\lambda_t}\,.
 \end{equation}
 To see this, note that there is a bijection between filtered tableaux on 
 $\mu$ with content $((a_1,\lambda_1),\dotsc,(a_t,\lambda_t))$ and 
 triples $(\nu,T_\bullet,T)$ where $\nu\subset\mu$, $T_\bullet$ is a
 filtered tableau of shape $\nu$ and content
 $((a_1,\lambda_1),\dotsc,(a_{t-1},\lambda_{t-1}))$, and $T$ is a Littlewood-Richardson
 tableau of shape $\mu/\nu$ and content $\lambda$;
 hence the number of these is 
\[
  \sum_{v \leq_{a_t} w} c^v \cdot c^{\mu/\nu}_{\lambda_t}\,.
\]

 But~\eqref{Eq:identity} follows from Proposition~\ref{P:identity}.
 Let $x$ (respectively $z$) be the permutation obtained from $v$ (respectively from $w$)
 by reversing the first $a_t$ values,
 i.e.
\[
  x(i)\ = \  %
  \begin{cases}
    v(a_t+1-i) &\text{if $1 \leq i\leq a_t$} \\
    v(i) &\text{otherwise.} 
  \end{cases}
\]
 Then $x$ and $z$ are Grassmannian permutations with descent $a_t$, and 
 shapes $\nu|_{a_t}$ and $\mu|_{a_t}$, respectively,
 and $\mu/\nu= (\mu|_{a_t})/(\nu|_{a_t})$.
 Furthermore, $x\leq_{a_t} z$ and $wv^{-1}=zx^{-1}$, from which we
 deduce~\eqref{Eq:identity}.
\end{proof}

%
%
%
\section{Further Remarks}\label{Sec:Comments}

When all the classes $\sigma_{\lambda_i}$ have 
degree 2 ($\lambda_i=\Box$, a single box),  
the multiplication formula $\sigma_w\cdot\pi^*_{a_i}\Box$ 
is due to Monk~\cite{Mo59}.  Monk's formula states that
\begin{equation}
\label{E:Monk}
  \sigma_w \cdot \pi^*_{a_i}\Box \ = 
  \sum_{\substack{j \leq i < k \\ \ell(w r_{jk}) = \ell(w)+1}}
  \sigma_{w r_{jk}}\,,
\end{equation}
where $r_{jk} \in S_n$ is the transposition swapping $j$ and $k$.
Iterating Monk's formula one sees that the coefficient of
$\ptclass_\alpha$ in a
product $\prod_{i=1}^{\dim(\alpha)} \pi_{a_i}^* \Box$ is
obtained by counting certain chains in the Bruhat order.  It is
not hard to see directly from~\eqref{E:Monk} that each permutation
$w$ in such a chain corresponds to a shape $\mu$ in 
$\Shift_\alpha$ such that the number of boxes in the column $j$ 
of $\mu$ equals $\#\{k \in [j]\ |\ w(k)>w(j+1)\}$, for 
all $j \in \{\min(\alpha), \dotsc, n-1\}$.  Indeed, if the permutation
$w$ does not correspond to a shape, then no
term on the right hand side of \eqref{E:Monk} corresponds to
a shape.
It follows that the coefficient is the number of chains of shapes 
in $\Shift_\alpha$
where the $i$th step involves adding a box in $\Rect_{a_i}$, which
is the answer given by our formula.

For example, we have
\[
    2\ptclass_{[3]}\ =\ 
   \pi^*_1 \includegraphics{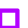}\cdot
   \pi^*_1 \includegraphics{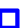}\cdot
   \pi^*_2 \includegraphics{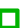}\cdot
   \pi^*_2 \includegraphics{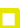}\cdot
   \pi^*_3 \includegraphics{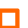}\cdot
   \pi^*_3 \includegraphics{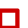}\ ,
\]
as there are two chains of shapes which satisfy this condition.
\[
  \includegraphics[height=30pt]{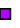}\ \ 
     \raisebox{21pt}{$\subset$}\ \ 
  \includegraphics[height=30pt]{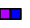}\ \ 
     \raisebox{21pt}{$\subset$}\ \ 
  \includegraphics[height=30pt]{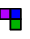}\ \ 
     \raisebox{21pt}{$\subset$}\ \ 
  \includegraphics[height=30pt]{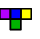}\ \ 
     \raisebox{21pt}{$\subset$}\ \ 
  \includegraphics[height=30pt]{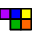}\ \ 
     \raisebox{21pt}{$\subset$}\ \ 
  \includegraphics[height=30pt]{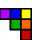}
\]
\[
  \includegraphics[height=30pt]{figures/M1.eps}\ \ 
     \raisebox{21pt}{$\subset$}\ \ 
  \includegraphics[height=30pt]{figures/M2.eps}\ \ 
     \raisebox{21pt}{$\subset$}\ \ 
  \includegraphics[height=30pt]{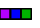}\ \ 
     \raisebox{21pt}{$\subset$}\ \ 
  \includegraphics[height=30pt]{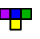}\ \ 
     \raisebox{21pt}{$\subset$}\ \ 
  \includegraphics[height=30pt]{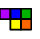}\ \ 
     \raisebox{21pt}{$\subset$}\ \ 
  \includegraphics[height=30pt]{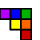}\smallskip
\]

It is possible to give a purely combinatorial proof of
Theorem~\ref{Thm:Main} using Kogan's formula~\cite[Theorem 2.4]{Ko01}.
This rule is based on insertion of RC-graphs and gives the coefficient 
$c^w_{v,a,\lambda}$,  
when $v(a{+}1)<v(a{+}2)<\dotsb<v(n)$.
In particular, this gives a formula for the product when $v$ and $w$ are a valley
permutations with a floor at $a$, and so we may use this in a 
formula for the intersection numbers of Theorem~\ref{Thm:Main}
to give a combinatorial proof. 

The conventions in~\cite{Ko01} for Schubert classes differ from those used
in this article.
To compare conventions, it is necessary to replace our permutations $w$ by
$\widetilde{w}=w_0 w w_0$ throughout.
In particular, a cohomology class indexed by $w$ in this article is the class
indexed by $\widetilde{w}$ in~\cite{Ko01}.
Thus our condition on $v$ becomes
$\widetilde{v}(1)<\widetilde{v}(2)<\dotsb<\widetilde{v}(a)$, which is the 
condition found in~\cite{Ko01}.

To deduce Theorem~\ref{Thm:Main} from this formula, we would need to show that,
for valley permutations $w,v$ with floor at $a$, Kogan's rule for $c^w_{v,a,\lambda}$
coincides with the Littlewood-Richardson rule for $c^{\mu/\nu}_\lambda$,
where $\nu=\mu(v)|_a$ and $\mu=\mu(w)|_a$.
Here, $\mu(v)$ is the shape of $v$ and $\mu(w)$ is the shape of $w$.
While this is certainly possible, we chose not to pursue this.

%
%
%
%
%
%
\appendix
\section{More examples}

\begin{example}
Consider the following product in $\Fl_{235}$,
\[
   \pi_2^*(\sigma_{\includegraphics[height=4pt]{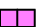}})\cdot
   \pi_2^*(\sigma_{\includegraphics[height=4pt]{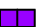}})\cdot
   \pi_3^*(\sigma_{\raisebox{-4pt}{\includegraphics[height=8pt]{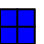}}})\cdot
   \pi_3^*(\sigma_{\raisebox{-4pt}{\includegraphics[height=8pt]{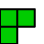}}})\cdot
   \pi_5^*(\sigma_{\includegraphics[height=4pt]{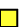}})\cdot
   \pi_5^*(\sigma_{\raisebox{-8pt}{\includegraphics[height=12pt]{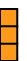}}})\cdot
   \pi_5^*(\sigma_{\raisebox{-8pt}{\includegraphics[height=12pt]{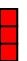}}})\,.
\]

By Theorem~\ref{Thm:Main}, the coefficient of $\ptclass$
is the number of filtered tableau with content
$((2,\includegraphics[height=6pt]{appendix/2.1.eps}), 
  (2,\includegraphics[height=6pt]{appendix/2.2.eps}),
  (3,\raisebox{-2pt}{\includegraphics[height=12pt]{appendix/22.3.eps}}),
  (3,\raisebox{-2pt}{\includegraphics[height=12pt]{appendix/21.4.eps}}),
  (5,\includegraphics[height=6pt]{appendix/1.5.eps}),
  (5,\raisebox{-6pt}{\includegraphics[height=18pt]{appendix/111.6.eps}}),
  (5,\raisebox{-6pt}{\includegraphics[height=18pt]{appendix/111.7.eps}})),
$
which is 18:

\[
  \begin{picture}(70,70)
    \put(0,0){\includegraphics[height=70pt]{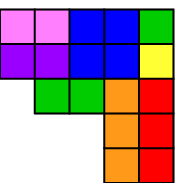}}
  \end{picture}
      \qquad
  \begin{picture}(70,70)
    \put(0,0){\includegraphics[height=70pt]{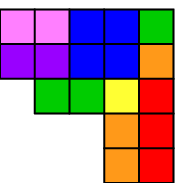}}
  \end{picture}
      \qquad
  \begin{picture}(70,70)
    \put(0,0){\includegraphics[height=70pt]{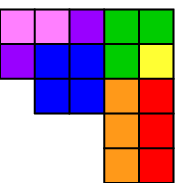}}
  \end{picture}
      \qquad
  \begin{picture}(70,70)
    \put(0,0){\includegraphics[height=70pt]{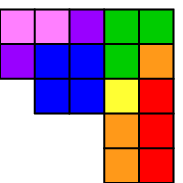}}
  \end{picture}
     \qquad
  \begin{picture}(70,70)
    \put(0,0){\includegraphics[height=70pt]{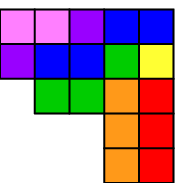}}
  \end{picture}
\]

\[
  \begin{picture}(70,70)
    \put(0,0){\includegraphics[height=70pt]{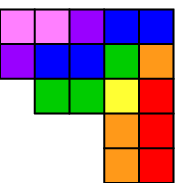}}
  \end{picture}
      \qquad
  \begin{picture}(70,70)
    \put(0,0){\includegraphics[height=70pt]{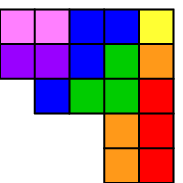}}
  \end{picture}
      \qquad
  \begin{picture}(70,70)
    \put(0,0){\includegraphics[height=70pt]{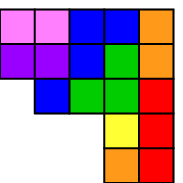}}
  \end{picture}
      \qquad
  \begin{picture}(70,70)
    \put(0,0){\includegraphics[height=70pt]{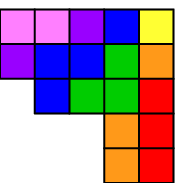}}
  \end{picture}
      \qquad
  \begin{picture}(70,70)
    \put(0,0){\includegraphics[height=70pt]{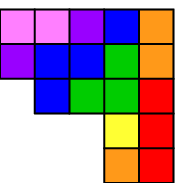}}
  \end{picture}
\]

\[
  \begin{picture}(70,70)
    \put(0,0){\includegraphics[height=70pt]{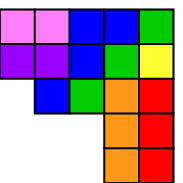}}
    \put(31.5,31){2} \put(45.5,44.5){1} \put(58.5,58){1}
  \end{picture}
      \qquad
  \begin{picture}(70,70)
    \put(0,0){\includegraphics[height=70pt]{appendix/T5.eps}}
    \put(32,31){1} \put(45,44.5){2} \put(58.5,58){1}
  \end{picture}
      \qquad
  \begin{picture}(70,70)
    \put(0,0){\includegraphics[height=70pt]{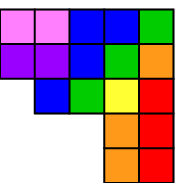}}
    \put(31.5,31){2} \put(45.5,44.5){1} \put(58.5,58){1}
  \end{picture}
     \qquad
  \begin{picture}(70,70)
    \put(0,0){\includegraphics[height=70pt]{appendix/T6.eps}} 
    \put(32,31){1} \put(45,44.5){2} \put(58.5,58){1}
  \end{picture}
\]

\[
  \begin{picture}(70,70)
    \put(0,0){\includegraphics[height=70pt]{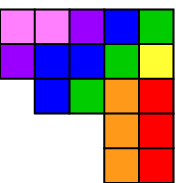}}
    \put(31.5,31){2} \put(45.5,44.5){1} \put(58.5,58){1}
  \end{picture}
      \qquad
  \begin{picture}(70,70)
    \put(0,0){\includegraphics[height=70pt]{appendix/T7.eps}}
    \put(32,31){1} \put(45,44.5){2} \put(58.5,58){1}
  \end{picture}
      \qquad
  \begin{picture}(70,70)
    \put(0,0){\includegraphics[height=70pt]{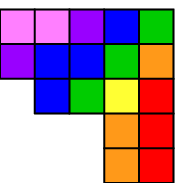}}
    \put(31.5,31){2} \put(45.5,44.5){1} \put(58.5,58){1}
  \end{picture}
      \qquad
  \begin{picture}(70,70)
    \put(0,0){\includegraphics[height=70pt]{appendix/T8.eps}}
    \put(32,31){1} \put(45,44.5){2} \put(58.5,58){1}
  \end{picture}
\]

\end{example}

\begin{minipage}[t]{2.2in}
\mbox{\ }

\begin{example}
  We remarked in Section~\ref{Sec:Comments} that, when every partition 
  is a single box ($\lambda_i=\Box$), a filtered tableau is a particular
  saturated chain of shapes in $\Shift_\alpha$.
  When $n=6$ we look at this for the problem 
\[
   (\pi^*_2 \includegraphics{figures/b2.eps})^4 \cdot
   (\pi^*_3 \includegraphics{figures/b3.eps})^5\cdot
   (\pi^*_4 \includegraphics{figures/b6.eps})^4
\]
 in $\Fl_{234}$.

 \quad To the right is the poset of shapes $\mu$ in $\Shift_{234}$,
 where at level $t$ (from the top) the shape has at most 
 $a_t$ and at least $a_t{-}1$ rows.

\quad Further to the right, we count the number of chains in this poset,
 which shows that the intersection number is 262.
\end{example}

\end{minipage}
\qquad
\begin{minipage}[t]{3.1in}
\[
  \begin{picture}(100,350)(0,-4)
   \put(10,-4){\includegraphics{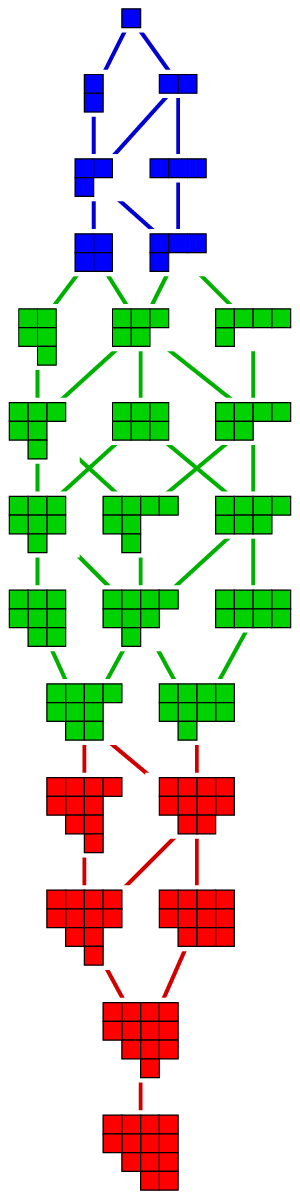}}
  \end{picture}
  \qquad
  \begin{picture}(100,350)(0,-4)
   \put(10,-4){\includegraphics{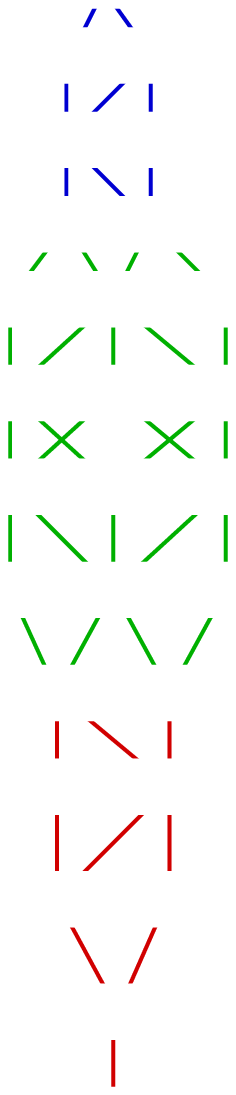}}
                    \put(45,327){1}
         \put(31,307){1}     \put(56,307){1}
         \put(31,283){2}     \put(56,283){1}
         \put(31,258){2}     \put(56,258){3} 
     \put(15,236){2} \put(45,236){5} \put(76,236){3}
     \put(15,209){7} \put(45,209){5} \put(76,209){8}
     \put(12,182){12}\put(42,182){15}\put(72,182){13}
     \put(12,152){12}\put(42,152){40}\put(72,152){13}
        \put(25,122){52}          \put(58,122){53}
        \put(25, 95){52}          \put(55, 95){105}
        \put(22, 62){157}         \put(55, 62){105}
                    \put(39, 31){262}
                    \put(39,  0){262}
  \end{picture} 
\]
\end{minipage}


\begin{thebibliography}{RSSS06}

\bibitem[BS98]{BS98}
N.~Bergeron and F.~Sottile, \emph{Schubert polynomials, the {B}ruhat order, and
  the geometry of flag manifolds}, Duke Math. J. \textbf{95} (1998), no.~2,
  373--423.

\bibitem[Co07]{Co07}
I.~Coskun, \emph{A Littlewood-Richardson rule for two-step flag varieties}, 
  Mss., 2007.

\bibitem[Ful97]{Fu97}
W. Fulton, \emph{Young tableaux}, Cambridge University Press, Cambridge, 1997,
  With applications to representation theory and geometry.

\bibitem[Knu00]{Kn00}
A.~Knutson, \emph{Descent cycling in {S}chubert calculus}, Experiment. Math.
  \textbf{10} (2000), no.~3, 345--353.

\bibitem[Kog01]{Ko01}
M.~Kogan, \emph{R{C}-graphs and a generalized {L}ittlewood-{R}ichardson rule},
  Internat. Math. Res. Notices (2001), no.~15, 765--782.

\bibitem[Mon59]{Mo59}
D.~Monk, \emph{The geometry of flag manifolds}, Proc. London Math. Soc. (3)
  \textbf{9} (1959), 253--286.

\bibitem[RSSS06]{RSSS06}
J.~Ruffo, Y.~Sivan, E.~Soprunova, and F.~Sottile, \emph{Experimentation and
  conjectures in the real schubert calculus for flag manifolds}, Experiment.
  Math. \textbf{15} (2006), no.~2, 199--221.

\bibitem[Sot96]{So96}
F.~Sottile, \emph{Pieri's formula for flag manifolds and {S}chubert
  polynomials}, Ann. Inst. Fourier (Grenoble) \textbf{46} (1996), no.~1,
  89--110.

\bibitem[Sta00]{St00}
R.~P. Stanley, \emph{Positivity problems and conjectures in algebraic
  combinatorics}, Mathematics: frontiers and perspectives, Amer. Math. Soc.,
  Providence, RI, 2000, pp.~295--319.

\end{thebibliography}
\providecommand{\bysame}{\leavevmode\hbox to3em{\hrulefill}\thinspace}
\providecommand{\MR}{\relax\ifhmode\unskip\space\fi MR }
\providecommand{\MRhref}[2]{%
  \href{http://www.ams.org/mathscinet-getitem?mr=#1}{#2}
}
\providecommand{\href}[2]{#2}

\end{document}